\documentclass[journal]{IEEEtran}
\usepackage{cite}

\ifCLASSINFOpdf
   \usepackage[pdftex]{graphicx}
\else
\fi
\usepackage{amsmath,amssymb,amsfonts}
\usepackage{algorithmic}
\usepackage{amsthm}

\usepackage{subcaption}
\usepackage{color}

\newtheorem{definition}{Definition}
\newtheorem{proposition}{Proposition}
\newtheorem{remark}{Remark}

\newcommand{\D}{\mathcal{D}}
\newcommand{\M}{\mathcal{M}}
\newcommand{\DMM}{\left[\partial \mathcal{M}\right]_-}

\newcommand{\A}{\mathcal{A}}

\newcommand{\DAM}{\left[\partial \mathcal{A}\right]_-}

\newcommand{\RR}{\mathbb{R}}

\begin{document}
%
\title{Transient Stability Analysis of Power Grids with Admissible and Maximal Robust Positively Invariant Sets}
%
%
%


\author{Tim Aschenbruck, Willem Esterhuizen, and Stefan Streif%
\thanks{This project has received funding from the European Social Fund (ESF) Grant Number: 100327773. This work was partially funded by BMBF-Projekt 05M18OCA: "Verbundprojekt 05M2018 - KONSENS:Konsistente Optimierung und Stabilisierung elektrischer Netzwerksysteme".
	 (Corresponding author: S. Streif)}
\thanks{The authors are with the Technische Universit\"{a}t Chemnitz, Automatic Control and System Dynamics Laboratory (e-mail: tim.aschenbruck@etit.tu-chemnitz.de; willem.esterhuizen@etit.tu-chemnitz.de; stefan.streif@etit.tu-chemnitz.de). }
}

\maketitle

\begin{abstract}
The energy transition is causing many stability-related challenges for power systems.
Transient stability refers to the ability of a power grid's bus angles to retain synchronism after the occurrence of a major fault.
In this paper a set-based approach is presented to assess the transient stability of power systems. 
The approach is based on the theory of \emph{barriers}, to obtain an exact description of the boundaries of admissible sets and maximal robust positively invariant sets, respectively.
We decompose a power system into generator and load components, replace couplings with bounded disturbances and obtain the sets for each component separately.
From this we deduce transient stability properties for the entire system.
We demonstrate the results of our approach through an example of one machine connected to one load and a multi-machine system.
\end{abstract}


%
\IEEEpeerreviewmaketitle

\section{Introduction}
\label{sec:introduction}
Safe grid operation requires various notions of stability to be enforced, such as voltage, frequency and rotor-angle stability, see for example \cite{kundur_etal_2004_definitionPSStability}. The \emph{transient stability problem} is concerned with the capability of a power system to retain synchronism after being subjected to a major contingency, for example, short circuits or faults in grid buses which lead to a loss of generation or a large load \cite[p.827]{kundur1994power_BOOK}.

For power system operators it is important to know which contingencies will lead to serious consequences. 
Typically, this will be checked with a long list of contingencies which need to be classified as either needing corrective action or not.
This may be achieved with a transient stability assessment system (TSA), which has two components: \emph{dynamic contingency screening} (DCS) where ``safe'' (or ``stable'') contingencies are screened out, and \emph{detailed time-domain stability analysis} where the remaining contingencies are simulated in detail over a time-interval, and then classified. For more information on TSAs the reader is referred to the references \cite{chiang2010onlineTSscreening} and \cite[p.447]{chiang2011direct_BOOK}.

There exist a number of approaches to perform DCS, including \emph{time-domain simulation} over a short time interval to identify highly unstable contingencies but they can not provide general stability limits for a safe operation of the system. 
\emph{Direct methods}, where stability arguments are made using Lyapunov functions or ``energy functions'', as in \cite{varaiya1985direct,ribbens1985direct_methods_A_SURVEY,Vittal_1983TS_EnergyFuntions,chiang1989closestUEP,Turitsyn16_LFfam,Turitsyn17_Aframework,Turitsyn2018robustnessMeasure,Althoff2016TS_opt_rational_LF}.
\emph{Hybrid approaches} combining direct methods and time-domain simulations \cite{Hagenmeyer2019}
and \emph{set-based methods}, such as  \cite{althoff2012TS_reachableSets,Althoff2017reachableSets,Turitsyn2017dynamic_fwReachableSets}, where forward reachable sets are found for the ``post-fault'' dynamics, and the references \cite{susuki2012hybridsys_TS_reachable_analysis,jin2010reachability} where backwards reachable sets are found that constitute regions from where the state is guaranteed to reach a neighbourhood of an equilibrium point. 
Another set-based method is presented in \cite{Henrion_2018} where invariant sets for synchronous machines are computed using the idea of \emph{occupation measures}, and it is shown that inner approximations may be found through the solution of semi-definite programs.

In this work, we propose a new set-based DCS approach to analyse and assess the transient stability of power systems through the computation of \emph{maximal robust positively invariant} (MRPI) and \emph{admissible} sets, using the theory of \emph{barriers in constrained nonlinear systems} \cite{DeDona_Levine2013barriers,ESTERHUIZEN_2016,Ester_Asch_Streif_2019,aschenbruck2020sustainability}.

%
We consider a power grid (modelled as a graph) and decouple it into load and generator nodes. 
Then we consider each node individually, treating the states of neighbouring nodes as bounded disturbance inputs.
Afterwards we compute the sets 
for each individual node (which is one or two dimensional), working around the problem of extremely large dimension. 
As will be shown, these sets are associated with \emph{safe}, \emph{potentially safe}, and \emph{unsafe} post-fault states. 
Transient stability can then be checked through the location of the grid component post-fault state coordinates that are associated with a \emph{safe}, \emph{potentially safe}, or \emph{unsafe} operation.
%

The used model for generator nodes is a second order ODE called the \emph{swing equation} and a standard model in the transient stability assessment.
Along with the decomposition, this second order model allows us to exploit some advantages of the used set based method, like the exact description of the boundaries of the sets.
For higher dimensions the computation of the sets is difficult and the manifold defining the boundaries are approximations based on a finite number of computed trajectories.

The outline of the paper is as follows. 
Section~\ref{sec_probsetup} introduces the transient stability problem in power grids and the power system models.
In Section~\ref{sec_newsetapproach} we present the new set-based approach for the transient stability assessment using three different sets: the largest set of safe operation, a set of potentially safe operations, and a set of unsafe operations which lead to stability problems. 
We summarize the theory to construct the MRPI and admissible set for constrained nonlinear systems and introduce the decomposition of the power grid model into its components.
This allows a simple verification of safe, potentially safe or unsafe operation for each node via the measurement of the state and its location with respect to the various sets.
Section~\ref{sec_detailed_analysis_gens} presents our main result with the detailed analysis of the introduced sets for the grid component dynamics describing generator and load nodes. 
In Section~\ref{section:examples} we provide two examples: one for a two-bus system, and one for a multi-machine-infinite-bus system. Section~\ref{section:conclusion} concludes the paper.


\section{Problem Setup} \label{sec_probsetup}

\subsection{The transient stability problem of power grids}

The power grid's stability is classified into three main issues: rotor angle, frequency, and voltage stability \cite{kundur_etal_2004_definitionPSStability}.
Transient stability concerns the rotor angle stability for short-term dynamics and describes the ability of a system to retain synchronism after a large disturbance occurs\cite[p.827]{kundur1994power_BOOK}, \cite[p.19]{chiang2011direct_BOOK},\cite{varaiya1985direct}.
Synchronicity of power systems means that the resulting angular difference between all its machines remains within certain bounds after such a large disturbance, which are usually called \emph{contingencies} and include for example different types of short circuits or the loss of generation and load buses.
Protective measures are often engaged after a contingency, altering the dynamics.
Thus, the mathematical description can be divided into three cases described by a set of three differential equation as in \cite{varaiya1985direct}
\begin{alignat}{2}
\dot{x} &= f_I(x),  &&\quad t \in\; ]-\infty,t_F[, \nonumber\\
\dot{x} &= f_F(x),  &&\quad t \in [t_F,t_C[, \nonumber\\  
\dot{x} &= f(x),    &&\quad t   \in [t_C,\infty[,\nonumber
\end{alignat}
where $f_I(x)$ denotes the initial dynamics before a fault occurs, called the \emph{pre-fault} dynamics, the time when the fault occurs is denoted by $t_F$ and the corresponding \emph{fault-on} dynamics by $f_F(x)$.
The fault is cleared from time $t_C$ and the system behaves according to the  \emph{post-fault} dynamics $f(x)$. 
Classically, transient stability is said to be guaranteed if the evolving trajectories of the fault-on dynamics remain in the  domain of attraction of an asymptotically stable equilibrium point of the post-fault dynamics\cite[p.5]{pai2012energy_BOOK}. 


\subsection{Power system model}

Power systems are networks with different components such as generators and loads, connected via transmission lines.
These systems can be presented as a network of coupled oscillators \cite{dorfler2013synchronization}, where the generator dynamics is represented as a second order ODE and the load dynamics as a first order ODE.
This representation of a power grid is based on the structure-preserving model \cite{bergen1981structPresModel}.
We consider the grid as an undirected graph where generator and load buses are the nodes connected via transmission lines as edges.
Let $\mathcal{G}$ and $\mathcal{L}$ be the set of indices of all generator and load nodes, respectively, with $|\mathcal{G}| = p_G$, $|\mathcal{L}| = p_L$, and let $p= p_G + p_L$.
For each generator, the dynamics is described by the \emph{swing equation}, which is often used for transient stability analysis \cite{Althoff2016TS_opt_rational_LF,Althoff2017reachableSets,chiang1989closestUEP,jin2010reachability,Henrion_2018,Turitsyn16_LFfam,Turitsyn17_Aframework,Turitsyn2018robustnessMeasure}, defined by
\begin{align}
m_i\ddot{\delta}_i + k_i\dot{\delta}_i + \sum_{j\in \mathcal{N}_i} a_{ij}\sin\left(\delta_i - \delta_j\right)  &= P_{m_i},  i\in\mathcal{G}, \label{eq_generator}
\end{align}
where $m_i$, $k_i$, $P_{m_i}$, refers to the $i$-th machine's, rotor mass, damping coefficient, and mechanical input torque, respectively.
The constant $a_{ij}$ is defined as $a_{ij} \triangleq V_i V_j B_{ij}$, where $V_i$ and $V_j$ are the bus voltages of the $i$-th and $j$-th bus, respectively, and $B_{ij}$ is the susceptance of the transmission line between node $i$ and $j$. 
Since the conductances are very small compared to the susceptances, we assume that the admittance of transmission lines only consists of its susceptances.
The symbol $\mathcal{N}_i$ denotes the indices of all neighbouring nodes of the $i$-th bus. 
The $i$-th machine's state variable $\delta_i$ is the error between the machine's rotor angle and  a reference angle from an infinite bus, which represents the connection to a large reference network, taken as $0$. We let $\omega_i \triangleq \dot{\delta}_i$.

The dynamics of each load node is described by the following first order ODE
\begin{align}
k_i\dot{\delta}_i + \sum_{j\in \mathcal{N}_i} a_{ij}\sin\left(\delta_i - \delta_j\right)  &= - P_{d_i},  i\in\mathcal{L}, \label{eq_load}
\end{align}
where, $P_{d_i}$ refers to the active power demand of the $i$-th load and $k_i$ is its damping coefficient. \textcolor{black}{The state varaible $\delta_i$ refers to the bus angle of the $i$-th load bus.}
All parameters $k_i$, $a_{ij}$, $m_i$, $P_{m_i}$, $P_{d_i}$ are nonnegative, and assumed to be constant.

\section{A new set-based approach for transient stability analysis} \label{sec_newsetapproach}
In this section we introduce the theory of barriers to construct the admissible set and the maximal robust positively invariant set and show how to utilise them for the power grid's transient stability analysis.
Power grids consist of many components, which makes the analysis of their large models vastly complicated.
Accordingly, we present an approach, in the second part of this section, to decompose the problem into many small problems.

\subsection{Admissible and maximal robust positively invariant sets}\label{subsec_Sets}

Here, we summarise the main results from \cite{DeDona_Levine2013barriers} and \cite{Ester_Asch_Streif_2019} for the admissible set and the MRPI, respectively. 
We consider a nonlinear system subjected to state and input constraints
\begin{subequations}
	\begin{align}
	\dot{x}(t) & = f(x(t),d(t)), \; \label{sys_eq_1} 
	x(t_0)  = x_0, \; 
	d  \in \mathcal{D}, \\ 
	g_i(x(t)) &\leq 0,\forall t\in[t_0,\infty[,\,\,i=1,2,\dots,p,\label{sys_eq_4}
	\end{align}
\end{subequations}
where $x(t)\in\RR^n$ denotes the state and $d(t)\in D \subset \RR^m$ denotes the disturbance input. We emphasise that the input $d$ will \emph{not} be considered as a control. 
We impose the same technical assumptions as those stated in \cite{DeDona_Levine2013barriers} and \cite{Ester_Asch_Streif_2019}, which are required for the rigorous analysis of the MRPI and admissible set:
\begin{description}
	\item[(A1)] The space $\mathcal{D}$ is the set of all Lebesgue measurable functions that map the interval $[t_0,\infty[$ to a set $D\subset\RR^m$, which is compact and convex.
	\item[(A2)] The function $f$ is $C^2$ with respect to $d\in D$, and for every $d$ in an open subset containing $D$, the function $f$ is $C^2$ with respect to $x\in\RR^n$.
	\item[(A3)] There exists a constant $0 < c < +\infty$ such that the following inequality holds true for all $x\in\RR^n$:
	$\sup_{d\in D}\vert x^T f(x,d) \vert \leq c(1+ \Vert x \Vert^{2} ).$
	\item[(A4)] The set $f(x,D)\triangleq\{f(x,d) : d\in D\}$ is convex for all $x\in \RR^{n}$.
	\item[(A5)] For every $i=1,2,\dots,p$, the function $g_i$ is $C^{2}$ with respect to $x\in\RR^n$, and the set $\{x:g_i(x) = 0\}$ defines a manifold.
\end{description}
In particular, the most important consequence of the assumptions is that the mentioned sets are closed.

The solution of \eqref{sys_eq_1} at time $t$ with initial condition $x_0 \in \RR^n$ and disturbance input   $d\in \mathcal{D}$ is denoted by $x^{(d,x_0,t_0)}(t)$ or, if the initial time or initial condition is clear from context, by $x^{(d,x_0)}$ or $x^{(d)}$, respectively.
We introduce the following sets 
\begin{align}
G\triangleq& \{x : g_i(x)\leq 0, \quad  \forall i \in \{1,2,...,p\}\},\nonumber\\
G_-\triangleq& \{x : g_i(x) <0, \quad  \forall i \in \{1,2,...,p\}\},\nonumber\\
G_0\triangleq& \{x : \exists i \in \{1,2,...,p\} \;\; \text{s.t.} \;\; g_i(x) =0\},\nonumber
\end{align}
and refer to $G$ as the constrained state-space. Sometimes we will write ``$\forall t$'' to mean for all $t\in[t_0,\infty[$ to lighten our notation. We denote by  $L_fg(x,d)\triangleq \nabla g(x)f(x,d)$ the Lie derivative of a continuously differentiable function $g:\mathbb{R}^n \rightarrow \mathbb{R}$ with respect to $f(x,d)$ at the point $x$. 
The set of all indices $\{i\in\{1,2,\dots,p\} : g_i(x) = 0\}$ is denoted by  $\mathbb{I}(x)$.
Given a set $S$, $|S|$ denotes its cardinality and $S^\mathsf{C}$ its complement. The function $\mathsf{sat}(x,\overline{x},\underline{x})$ denotes the saturation function with upper and lower bounds specified by $\overline{x}$ and $\underline{x}$, respectively.
%
%

\begin{definition}
	The \emph{admissible set}\footnote{also called the viability kernel \cite{aubin2009viability}} of the system \eqref{sys_eq_1} - \eqref{sys_eq_4}, denoted by $\A$, is the set of initial states for which there exists a $d\in \D$ such that the corresponding solution to \eqref{sys_eq_1} satisfies the constraints \eqref{sys_eq_4} for all future time.
	\begin{align*}
	\mathcal{A} \triangleq \left\{ x_0 \in \RR^n: \exists d\in\mathcal{D}, \;\;  x^{(d,x_0,t_0)}(t)\in G \;\; \forall t \in [t_0,\infty[ \;\right\}.
	\end{align*}
	\vspace{-.5cm}
\end{definition}
\begin{definition}
	A set $\Omega\subset\RR^n$ is said to be a \emph{robust positively invariant set} (RPI) of the system \eqref{sys_eq_1} provided that $x^{(d,x_0,t_0)}(t)\in\Omega$ for all $t\in[t_0,\infty[$, for all $x_0\in\Omega$ and for all $d\in\D$.
\end{definition}
\begin{definition}
	The \emph{maximal robust positively invariant set} (MRPI) of the system \eqref{sys_eq_1}-\eqref{sys_eq_4} contained in $G$, is the union of all RPIs that are subsets of $G$. Equivalently
	\begin{align*}
	\mathcal{M} \triangleq \left\{ x_0\in\RR^n: x^{(d,x_0,t_0)}(t)\in G, \,\,\forall d\in\mathcal{D}, \;\;  \forall t \in [t_0,\infty[ \;\right\},
	\end{align*}
	\vspace{-.5cm}
\end{definition}
this equivalence is proved in \cite[Proposition~2]{Ester_Asch_Streif_2019}.
As elaborated on in \cite{DeDona_Levine2013barriers} and \cite{Ester_Asch_Streif_2019}, under the assumptions, which are needed to have compactness of the space of solutions, $\A$ and $\M$ are closed. We denote their boundaries by $\partial \mathcal{A}$ and $\partial \mathcal{M}$, and define the following sets: $\DAM \triangleq  \partial\mathcal{A}\cap G_-$ and $\DMM  \triangleq \partial\mathcal{M}\cap G_-.$
The sets $\DAM$ and $\DMM$ are called the \emph{barrier} and \emph{invariance barrier}, respectively, because they posses a so-called \emph{semi-permeability} property: if the state crosses $\DAM$, having initiated from $\A$, then it cannot re-enter $\A$ before first violating a constraint. Similarly, if the state crosses $\DMM$, having initiated from the complement of $\M$, it can never leave $\M$.
We now summarise the main results which we will use to construct the sets $\A$ and $\M$.
Under the Assumptions (A1) - (A5), for every initial condition on the invariance barrier, $\bar{x}\in\DMM$ (resp. barrier, $\bar{x}\in\DAM$) there exists an input $\bar{d}\in\D$ such that the resulting integral curve, $x^{(\bar{d},\bar{x})}$, remains on the invariance barrier (resp. barrier) until it intersects $G_0$. 
Moreover, this integral curve along with its input $\bar{d}$ satisfies the following necessary conditions.
There exists a nonzero absolutely continuous maximal solution $\lambda^{\bar{d}}$ to the adjoint equation
\begin{align}
\label{eq_adj}
&\dot{\lambda}^{\bar{d}}(t) = -\left( \frac{\partial f}{\partial x}(x^{(\bar{d})}(t),\bar{d}(t)) \right)^T \lambda^{\bar{d}}(t),\nonumber\\
\quad &\lambda^{\bar{d}}(\bar{t}) = (\nabla g_{i^*}(z))^T,
\end{align}
with $L_fg_{i^*}(z,\bar{d}(\bar{t})) \triangleq \max_{i\in\mathbb{I}(z)} L_fg_i(z,\bar{d}(\bar{t}))$
which maximizes (resp. minimizes) the Hamiltonian for almost every $t \leq \bar{t}$ 
\begin{align}
\max_{d\in D}\{\lambda^{\bar{d}}(t)^Tf(x^{(\bar{d})}(t),d) \} = \lambda^{\bar{d}}(t)^T f(x^{(\bar{d})}(t),\bar{d}(t)) = 0\label{eq_Hamil_M}
\end{align}
\vspace{-3mm}
\begin{align}
\Big(\Big.\text{resp.}\quad\min_{d\in D}\{\lambda^{\bar{d}}(t)^Tf(x^{(\bar{d})}(t),d) \} &= \lambda^{\bar{d}}(t)^T f(x^{(\bar{d})}(t),\bar{d}(t)) \nonumber \\ &= 0\Big.\Big)\label{eq_Hamil_A}.
\end{align}

Furthermore, these integral curves running along the barrier $\DMM$ (resp. $\DAM$) with the associated disturbance $x^{(\bar{d})}$ intersect $G_0$ in finite time at $\bar{t}$. 
We label the intersection point $z \triangleq x^{(\bar{d},\bar{x},t_0)}(\bar{t})\in G_0$.
For this point, the following condition holds
\begin{align}
\quad \,\,\,\, \max_{d\in D} \max_{i\in\mathbb{I}(z)}L_f g_i(z,d) 
&= L_fg_{i^*}(z,\bar{d}(\bar{t})) = 0 \label{thm1_ult_tan_M}
\end{align}
\begin{align}
\Big(\Big.\text{resp.}\quad \min_{d\in D} \max_{i\in\mathbb{I}(z)}L_f g_i(z,d) 
&= L_fg_{i^*}(z,\bar{d}(\bar{t})) = 0\Big.\Big). \label{thm1_ult_tan_A}
\end{align}
Since the integral curve  $x^{(\bar{d})}$ intersects $G_0$ in a tangential manner, this condition is the so-called \emph{ultimate tangentiality condition}.

The proofs can be found in \cite{DeDona_Levine2013barriers} and \cite{Ester_Asch_Streif_2019}.
With the stated conditions we are able to construct the invariance barrier for the MRPI $\mathcal{M}$ (resp. the barrier for the admissible set $\mathcal{A}$). The steps for doing so are as follows:
\begin{itemize}
	\item Determine the final condition for the system \eqref{sys_eq_1} by using \eqref{thm1_ult_tan_M} (resp. \eqref{thm1_ult_tan_A}) to identify the points of ultimate tangentiality $z \triangleq x^{(\bar{d},\bar{x},t_0)}(\bar{t})\in G_0$, i.e. where the integral curve along $\DMM$ (resp. $\DAM$) intersects $G_0$ at time $ \bar{t}$. 
	\item The integral curves characterizing $\DMM$ (resp. $\DAM$) can then be obtained via backward integration of the system dynamics \eqref{sys_eq_1} and adjoint equation \eqref{eq_adj} with the disturbance realisation $\bar{d}$, which satisfies \eqref{eq_Hamil_M} (resp. \eqref{eq_Hamil_A}) for almost all time.
\end{itemize}
Detailed information for the construction of the sets can be found in \cite{Ester_Asch_Streif_2019}.
We note that the conditions of the theorem are necessary; therefore some trajectories or parts of it 
may need to be ignored. Thus, we will refer to trajectories obtained via the maximum-like principle as \emph{candidate barrier trajectories}. 
Finally, we emphasise that one advantage of constructing the sets with this method, in comparison with other set-based methods, is that 
these conditions give an exact description for an integral curve defining the boundary of the set.


\subsection{Approach to decompose large scale networked system}\label{subsec_approach}

In general the considered system for the transient stability analysis $\dot{x} = f(x)$ is a large scale system with many connected nodes representing the grid components.
The direct application of set-based methods to high dimensional systems is a hard task and often intractable \cite{ChemTomlin2018decompositionbrs}.
Therefore, the decomposition of these systems to analyse many small systems instead is a common approach.
Our power grid decomposition follows the principle as used in \cite{ChenTomlin2016decouplingdisturbances} and \cite{li2019statedecomp}, where the large system is decomposed into several subsystems $\dot{x}_i = f_i(x_i,d_i)$, with $d_i$ as disturbance input. This variable may be regarded as a \emph{decoupling variable} because in the overall power grid system the considered components are connected via the neighbouring state variable describing the angle of the corresponding bus, we decouple the components and replace the state variable with the disturbance input $d_i$.
Thus, we consider each generator and load node by itself along with its neighbours.
Hence, the entire grid size does not influence the dimension of the analysed systems, it only effects the number of terms of the sum in \eqref{eq_generator} and \eqref{eq_load}. 
Advantageously, the number of neighbours for a node is usually vastly smaller than the number of nodes in the entire grid.
\begin{remark}
	The space of disturbances $\D$ is the set of all Lebesgue measurable functions that map the interval $[t_0,\infty[$ to a compact and convex set $D\subset\mathbb{R}^m$.
	Therefore the disturbance can be any Lebesgue measurable function.
\end{remark}
For safe grid operation we impose state constraints, which bound the angular difference with  $\underline{\delta}_i \leq \delta_i(t)\leq \overline{\delta}_i$ for generator and load nodes, where $\underline{\delta}_i, \overline{\delta}_i,$ are scalar constants. 
We interpret the $\delta$ state variables of the $i$-th node's neighbours as a \emph{disturbance input}. 
More specifically, we consider the $i$-th node, along with the indices of its neighbours $\mathcal{N}_i$. 
For all $j\in\mathcal{N}_i$ we let $d_i^j\triangleq \delta_j$ and, to lighten our notation, let $d_i\in\mathbb{R}^{|\mathcal{N}_i|}$ denote the vector formed by stacking the $d_i^j$'s. 
Moreover, we let $D_i^j\triangleq[\underline{\delta}_j,\overline{\delta}_j]$, and let $D_i$ denote the Cartesian product of the $|\mathcal{N}_i|$ intervals. 
We then obtain $\M_i$ and $\A_i$ for node $i$. In the case of generator nodes (i.e. $i\in\mathcal{G}$) the problem is to find the sets for the system
\begin{align}
\dot{\delta}_i(t) & = \omega_i(t),\nonumber\\
\dot{\omega}_i(t) & = \frac{-k_i\omega_i(t) - \sum_{j\in \mathcal{N}_i} a_{ij}\sin\left(\delta_i(t) - d_i^j(t)\right)   + P_{m_i}}{m_i},\nonumber
\end{align}
with $\delta_i(t) \in [\underline{\delta}_i,\overline{\delta}_i]\; \forall t $ and
$d_i(t) \in D_i\;\forall t.$
In the case of load nodes the dynamics are described by \eqref{eq_load}.
We again emphasise that the ``disturbance inputs'' $d_i^j$ are different from the aforementioned contingencies.

We can then use the computed sets as follows. Let $x\triangleq~\hspace{-.15cm}(\delta_1,\omega_1,\dots,\delta_{p_G},\omega_{p_G},\delta_{p_G + 1},\dots,\delta_{p_G+p_L})^T$ denote the post-fault state of the overall system; $x_i \triangleq (\delta_i,\omega_i)^T$ for $i\in \mathcal{G}$; and $x_i \triangleq \delta_i$ for $i\in \mathcal{L}$. Then, the post-fault state, $x$, is:
\begin{itemize}
	\item \emph{safe} if $x_i\in\mathcal{M}_i$ for all $i\in\mathcal{G}\cup\mathcal{L}$,
	\item \emph{potentially safe} if $x_i\in\mathcal{A}_i$ for at least one $i\in\mathcal{G}\cup\mathcal{L}$, and $x_i\notin\mathcal{A}^{\mathsf{C}}_i$ for all $i\in\mathcal{G}\cup\mathcal{L}$,
	\item \emph{unsafe} if there exists an $i\in\mathcal{G}\cup\mathcal{L}$ such that $x_{i}\in\mathcal{A}^{\mathsf{C}}_i$.
\end{itemize}

In contrast to what is done classically in the transient stability analysis, the invariance property of $\M$ does not guarantee the convergence of a trajectory initiating in $\M$ to a stable equilibrium point. 
We guarantee that the trajectories of the post-fault dynamics initiating in $\M$ remain within acceptable bounds.
Furthermore, with $\A$ we obtain an additional region that potentially guarantees that the trajectories remain within acceptable bounds, depending on the disturbance input, whereas all trajectories initiating from the complement of $\A$ definitely lead to constraint violations.

An advantage of this approach is that, given a post-fault state which is either potentially safe or unsafe, ``critical'' nodes may be identified as being ones that may experience constraint violations, allowing an operator to focus its resources accordingly. 
We also note that only post-fault states classified as potentially safe would require detailed stability analysis, whereas those classified as safe or unsafe would not, because they will never violate the constraints or they will definitely, respectively.
The idea is illustrated in Fig. \ref{fig:concept_illustration}. 

\begin{figure}[h]
	\begin{center}\
		\includegraphics[width=3\linewidth,height=12cm,keepaspectratio]{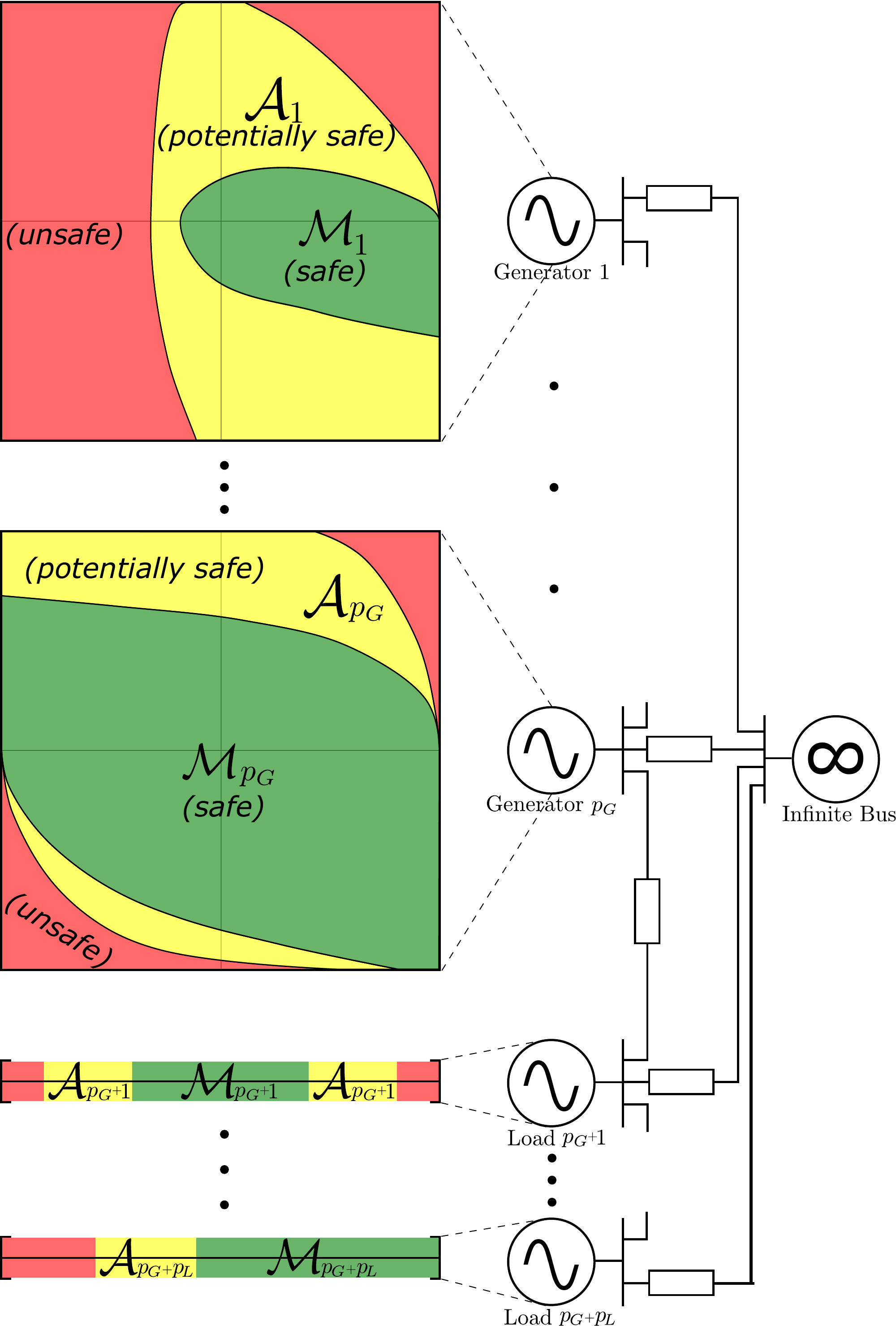} 
		\caption{Illustration of the set-based approach towards transient stability analysis} 
		\label{fig:concept_illustration}                             
	\end{center} 
	\vspace{-5mm}                               
\end{figure}

\section{Detailed analysis of the sets for decomposed power grid nodes}\label{sec_detailed_analysis_gens}

This section presents the main result of our paper, with the detailed analysis of the sets' construction for generator and load nodes. 
Furthermore we present new and simply verifiable conditions on the swing equations parameters under which it is guaranteed that candidate barrier trajectories exist.

\subsection{Generator nodes disturbance realisation and points of ultimate tangentiality} \label{subsec_G_dbarUTC}

Invoking \eqref{eq_Hamil_M} for the $i$-th generator 
we have that $\bar{d}_i$ associated with the MRPI's barrier should satisfy
\begin{align*}
&\lambda_i(t)^Tf(x_i(t),\bar{d}_i(t)) = \max_{d_i\in D_i}\left\{ \vphantom{\frac{-k_i\omega_i(t) - \sum_{j\in \mathcal{N}_i} a_{ij}\sin\left(\delta_i(t) - d_i^j\right) + P_{m_i}}{m_i}}	
\lambda_i^1(t)\omega_i(t) + \lambda_i^2(t) \right.\\
& \left. \left( 
\frac{-k_i\omega_i(t) - \sum_{j\in \mathcal{N}_i} a_{ij}\sin\left(\delta_i(t) - d_i^j\right) + P_{m_i}}{m_i}
\right)
\right\} = 0\nonumber
\end{align*}
for almost every $t$, where $\lambda_i\triangleq (\lambda_i^1, \lambda_i^2)^T$. We deduce that
\begin{align}
\bar{d}_i^j(t) = 
\begin{cases}
\label{dist_cases_invbarrier}
\mathsf{sat}(\delta_i(t) + \frac{\pi}{2},\overline{\delta}_j,\underline{\delta}_j) & \text{if} \quad \lambda_i^2(t)\geq 0,\\
\mathsf{sat}(\delta_i(t) - \frac{\pi}{2},\overline{\delta}_j,\underline{\delta}_j) & \text{if} \quad \lambda_i^2(t)< 0.
\end{cases}
\end{align}
Similarly from \eqref{eq_Hamil_A}, for $\bar{d}_i$ associated with $\DAM$ we have
\begin{align}
\bar{d}_i^j(t) = 
\begin{cases}
\label{dist_cases_barrier}
\mathsf{sat}(\delta_i(t) - \frac{\pi}{2},\overline{\delta}_j,\underline{\delta}_j) & \text{if} \quad \lambda_i^2(t)\geq 0,\\
\mathsf{sat}(\delta_i(t) + \frac{\pi}{2},\overline{\delta}_j,\underline{\delta}_j) & \text{if} \quad \lambda_i^2(t)< 0.
\end{cases}
\end{align}

We now turn our attention towards identifying candidate points of ultimate tangentiality located on the state constraint functions of the $i$-th generator node. Recall from Subsection~\ref{subsec_approach} that every generator node has two state constraint functions. 
We label them as follows: $g_i^1(x_i) = \delta_i - \overline{\delta}_i$ and $g_i^2(x_i) = -\delta_i + \underline{\delta}_i$.
We will label points of ultimate tangentiality $z_i \triangleq (\check{\delta}_i, \check{\omega}_i)$. 

Invoking condition \eqref{thm1_ult_tan_M} for the MRPI we get
\begin{align*}
&\max_{d_i\in D_i } \left\{\nabla g_i^1(z_i) f(z_i,d_i) \right\} = \max_{d_i\in D_i }\{\check{\omega}_i\} = \check{\omega}_i = 0, \\
&\max_{d_i\in D_i } \left\{\nabla g_i^2(z_i) f(z_i,d_i) \right\} = \max_{d_i\in D_i }\{-\check{\omega}_i\} = -\check{\omega}_i = 0,
\end{align*}
which identifies two points, $z_i = (\underline{\delta}_i,0)$ and $z_i = (\overline{\delta}_i,0)$. It is easy to show that these two points are also points of ultimate tangentiality for the admissible set.

\subsection{Existence of candidate barrier trajectories for generator nodes}
Recall that the stated conditions in Section \ref{sec_newsetapproach}
are necessary, and that some obtained integral curves or parts of them may not form parts of the barriers, $\DMM$ and $\DAM$, and need to be ignored. 
This is clearly the case for parts of integral curves that are outside $G$. In this section we provide conditions, in terms of a generator node's parameters, under which it is guaranteed that a trajectory evolves backwards into $G_-$, and thus may form part of a barrier. 
Furthermore, we briefly discuss what we call the ``bounce phenomenon'', which, when it occurs, implies that a trajectory is not a barrier trajectory.

First, we consider the existence of candidate trajectories evolving backwards into the interior of $G$.
Therefore, the next proposition is concerned with barrier trajectories ending on $g_i^1$ and $g_i^2$.
\begin{proposition}\label{prop:existence_traj_deltacstr}
	There exists a candidate barrier trajectory, associated with the MRPI or admissible set, partly contained in $G_-$ and ending at the point of ultimate tangentiality $(\overline{\delta}_i,0)^T$ (resp. $(\underline{\delta}_i,0)^T$), if and only if 
	\begin{align*}
	&\sum_{j\in\mathcal{N}_i} a_{ij}\sin\left(\overline{\delta}_i - \bar{d}_i^j(\bar{t})\right)  - P_{m_i} > 0\\\nonumber
	\bigg(\bigg.\text{resp.}\quad&\sum_{j\in\mathcal{N}_i} a_{ij}\sin\left(\underline{\delta}_i - \bar{d}_i^j(\bar{t})\right)  - P_{m_i} < 0\bigg.\bigg).
	\end{align*}
	\vspace{-.5mm}
\end{proposition}
\begin{proof}
	From \eqref{eq_adj}, the adjoint equation for machine $i\in\mathcal{G}$ is given by:
	$
	\dot{\lambda}_i^1(t)   = ( \sum_{j\in\mathcal{N}_i} \frac{a_{ij}}{m_i}\cos\left(\delta_i-\delta_j\right) ) {\lambda}_i^2(t)$ and $ 
	\dot{\lambda}_i^2(t)  = -{\lambda}_i^1(t) + \frac{k_i}{m_i} {\lambda}_i^2(t), 
	$
	with the final condition $\lambda_i(\bar{t}) = (1,0)^T$ associated with $g_i^1$. Thus, it follows that $\dot{\lambda}_i^2(\bar{t}) = -1$, which implies that $\lambda_i^2$ is increasing, going backwards in time from $\bar{t}$. Thus, $\lambda_i^2(t) \geq 0$ over some period of time before $\bar{t}$. A barrier trajectory along with the associated adjoint satisfies the Hamiltonian maximisation condition, \eqref{eq_Hamil_M}. Thus, the adjoint vector is orthogonal to the velocity vector all along the barrier trajectory. We can conclude that $\dot{\omega}(\bar{t}) < 0$, for otherwise the barrier trajectory would have approached $(\bar{\delta}_i,0)^T$ from the set $\{(\delta_i,\omega_i):\delta_i > \overline{\delta}_i\}$, which is outside $G$. The condition stated in the proposition then follows immediately from this, and adapts easily to the point $(\underline{\delta}_i,0)^T$.
\end{proof}
Note that the conditions of Propositions \ref{prop:existence_traj_deltacstr} does not depend on the mass or damping of the machine. 

Second, what we call the ``bounce phenomenon'' has been observed for some candidate trajectories.
Here, $\bar{d}_i^j$ may abruptly switch, causing the associated state trajectory to experience a ``bounce'', the resulting integral curve not clearly defining the boundary of a set and needing to be ignored. Let us denote the time at which a bounce occurs by $\hat{t}$. 
Recall that $\bar{d}_i^j$ is a function of $\lambda_i^2$, see equations \eqref{dist_cases_invbarrier} and \eqref {dist_cases_barrier}, and thus when this switch occurs $\lambda_i^2(\hat{t}) = 0$. 
Recall from 
the stated conditions in Section \ref{sec_newsetapproach}
that the adjoint is nonzero. 
Thus, when $\lambda_i^2(\hat{t}) = 0$, it is true that $\lambda_i^1(\hat{t}) \neq 0$, which implies, from \eqref{eq_Hamil_M} or \eqref{eq_Hamil_A}, that $\lambda_i^1(\hat{t})\omega_i(\hat{t}) = 0$, which implies $\omega_i(\hat{t}) = 0$. 
Thus, bounces can only occur if the candidate barrier trajectory intersects the $\delta_i$-axis and if $\dot{\omega}_i$ changes sign at $\hat{t}$. 

\subsection{Detailed analysis of the sets for load nodes}\label{sec_detailed_analysis_loads}
The analysis of finding $\M$ and $\A$ for loads is easier than for generators due the one-dimensional load dynamics \eqref{eq_load}. For a one-dimensional system the sets are connected. To see this, consider $\M$ and suppose it were not connected. Then there would exist an initial condition $x_0\in\M^\mathsf{C}$, along with a disturbance realisation, such that the resulting integral curve violates a constraint in the future, but which would first have to penetrate $\M$, which is impossible. A similar argument holds true for $\A$. Thus, to specify the sets we only need to find their lower and upper bounds. For the set $\M$ these are specified from the solution of the following two problems
\begin{align}
&\min_{\underline{\delta}_i \leq \delta_i\leq \overline{\delta}_i } \delta_i\quad \mathrm{subject}\,\,\mathrm{to}\quad \min_{d_i\in D_i } \dot{\delta}_i \geq 0, \label{eq_load_condlbM}\\
&\max_{\underline{\delta}_i \leq \delta_i\leq \overline{\delta}_i } \delta_i\quad \mathrm{subject}\,\,\mathrm{to}\quad \max_{d_i\in D_i } \dot{\delta}_i \leq 0.		   \label{eq_load_condubM}
\end{align}
For the set $\A$, they are obtained from
\begin{align}
&\min_{\underline{\delta}_i \leq \delta_i\leq \overline{\delta}_i } \delta_i\quad \mathrm{subject}\,\,\mathrm{to}\quad \max_{d_i\in D_i } \dot{\delta}_i \geq 0, \label{eq_load_condlbA}\\
&\max_{\underline{\delta}_i \leq \delta_i\leq \overline{\delta}_i } \delta_i\quad \mathrm{subject}\,\,\mathrm{to}\quad \min_{d_i\in D_i }\dot{\delta}_i \leq 0.		   \label{eq_load_condubA}
\end{align}

\section{Examples}\label{section:examples}

\subsection{Sets for a two-bus system}
%
To illustrate our approach in detail we consider a common one machine one load system as shown in Fig. \ref{fig:twobus_grid}.
\begin{figure}[h]
	\begin{center}\
		\includegraphics[width=7cm,height=7cm,keepaspectratio]{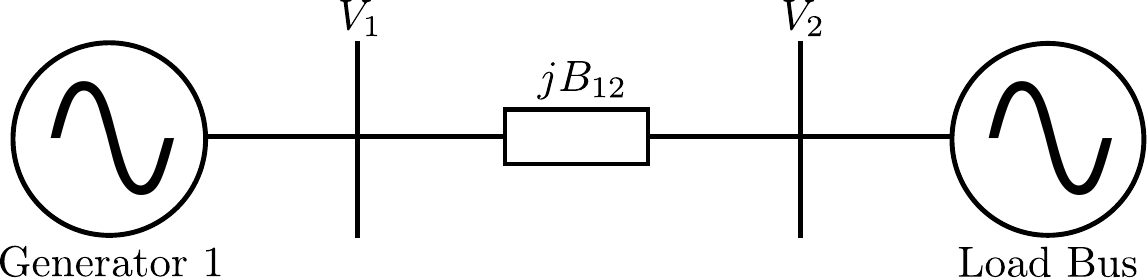} 
		\caption{Two-bus system with one generator and one load bus} 
		\label{fig:twobus_grid}                             
	\end{center}        
\end{figure}
The dynamics of the generator node is:
$\dot{\delta}_1  = \omega_1$, $
\dot{\omega}_1  = \frac{-k_1\omega_1 - a_{12}\sin\left(\delta_1-\delta_{2}\right) + P_{m_1}}{m_1},
$
and that of the load is:
$
\dot{\delta}_2 = \frac{-a_{12}\sin\left(\delta_2-\delta_1\right) - P_{d_2}}{k_2}.
$
We identify $x_1\triangleq (\delta_1,\omega_1)^T$, $x_2\triangleq \delta_2$, $d_1^2 \triangleq \delta_2$ and $d_2^1 \triangleq \delta_1$. 
First, we concentrate on the generator node.
We use the parameters as in \cite{Turitsyn16_LFfam}: $k_1=1$ p.u., $m_1=1$ p.u., $a_{12}=0.8$ p.u., and $P_{m_1}=0.4$ p.u.. 
We impose the constraints $|\delta_1|\leq \frac{\pi}{2}$ and assume $|\delta_2|\leq\bar{\delta}_2$. 
We proceed to find the barriers ending on the ultimate tangentiality points $(\pm\frac{\pi}{2},0)$ using \eqref{dist_cases_invbarrier} and \eqref{dist_cases_barrier}, assuming a different bound $\overline{\delta}_2$, as shown in Fig.~\ref{fig:twobus_A&M_G}. 
Clearly, each MRPI is a subset of its corresponding admissible set, this fact illustrated with different line types. 
We also note that $\M$ grows with decreasing $\overline{\delta}_2$, but $\A$ shrinks as this happens. When $\overline{\delta}_2 = 0$ (load is the \emph{infinite bus}) the sets coincide.
\begin{figure}
	\centering
	\begin{subfigure}[b]{1\linewidth}
		\centering
		\includegraphics[width=6.5\linewidth,height=6.5cm,keepaspectratio]{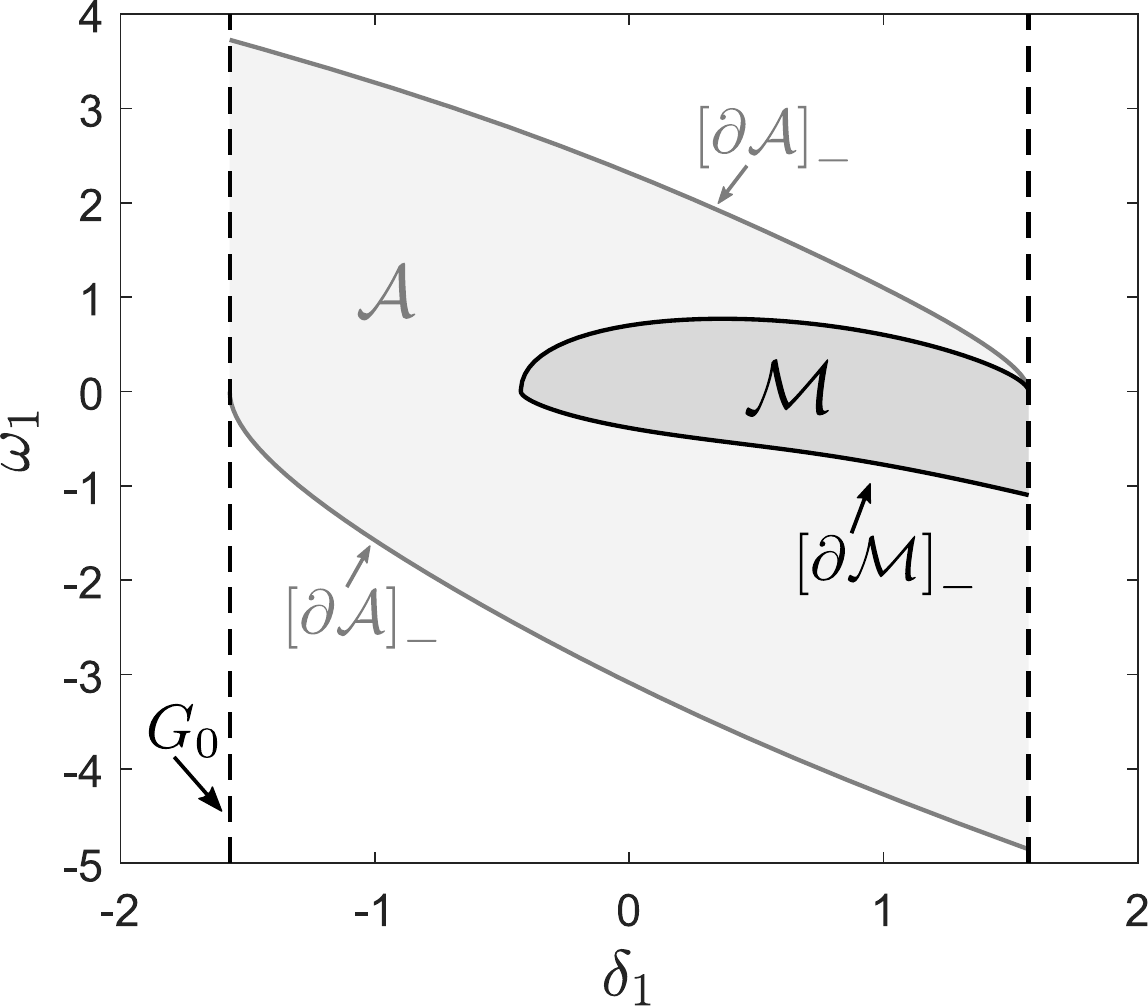} 
		\subcaption{ $\overline{\delta}_2=\frac{\pi}{3.7}$}
	\end{subfigure}
	\begin{subfigure}[b]{1\linewidth}
		\centering
		\vspace{3mm}
		\includegraphics[width=5\linewidth,height=6.5cm,keepaspectratio]{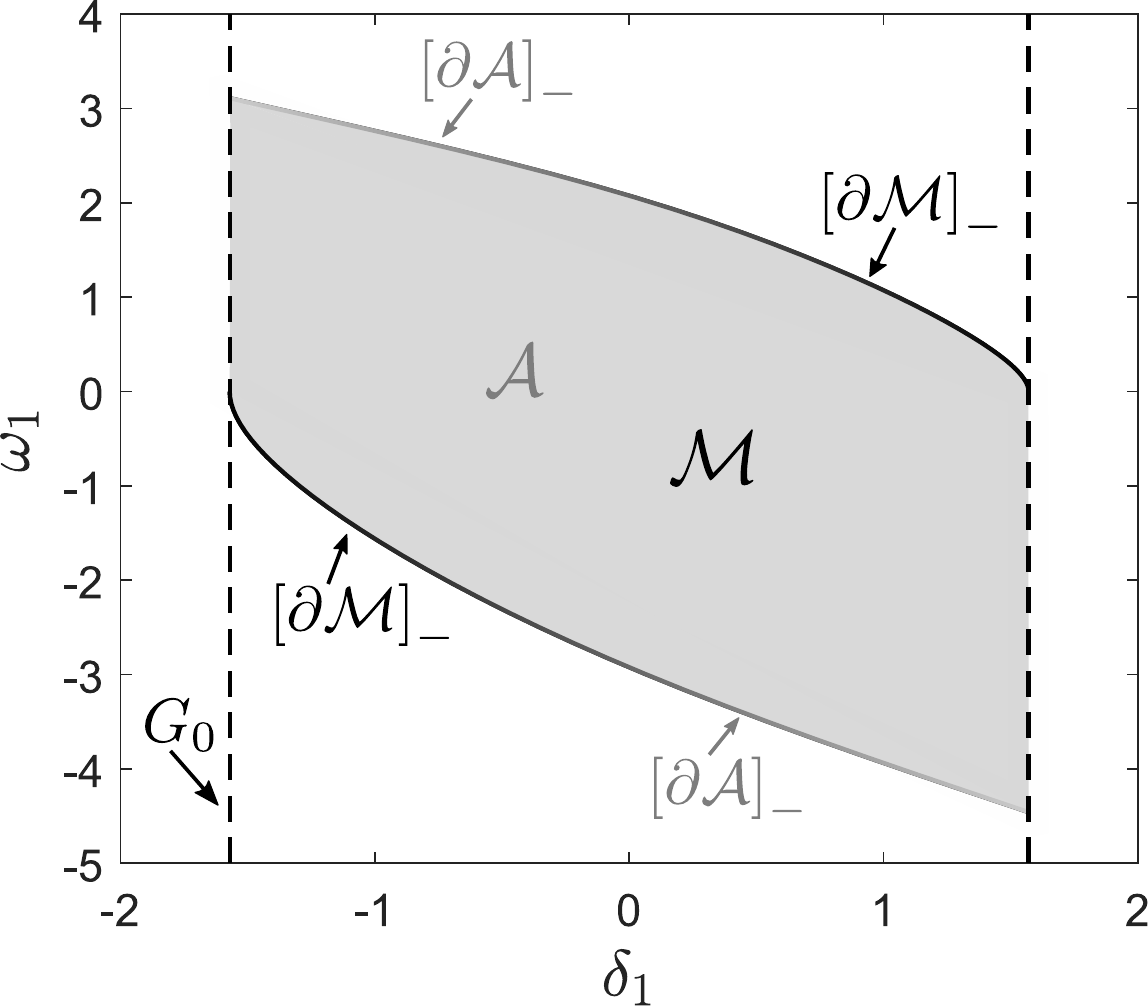} 
		\subcaption{$\overline{\delta}_2=0$}
	\end{subfigure}
	\caption{The sets for the two-bus system generator node with different disturbance bounds $\overline{\delta}_2$. The black line is the invariance barrier for the MRPI the gray line the barrier for the admissible set.}
	\label{fig:twobus_A&M_G}   
\end{figure}
Next, we concentrate on the load node with the parameters $k_2=1$ p.u., and $P_{d_2} = 0.7$ p.u.. The disturbance bound $\overline{\delta}_1=\frac{\pi}{2}$ and the different constraints on $\delta_2$ are given from before.
With \eqref{eq_load_condlbM} - \eqref{eq_load_condubA} we find the bounds for $\A$ and an upper bound for $\M$ within the constrained state space. 
The resulting sets for the load are shown in Fig. \ref{fig:twobus_A&M_L}.
\begin{figure}
	\centering
	\begin{subfigure}[b]{1\linewidth}
		\centering
		\includegraphics[width=7cm,keepaspectratio]{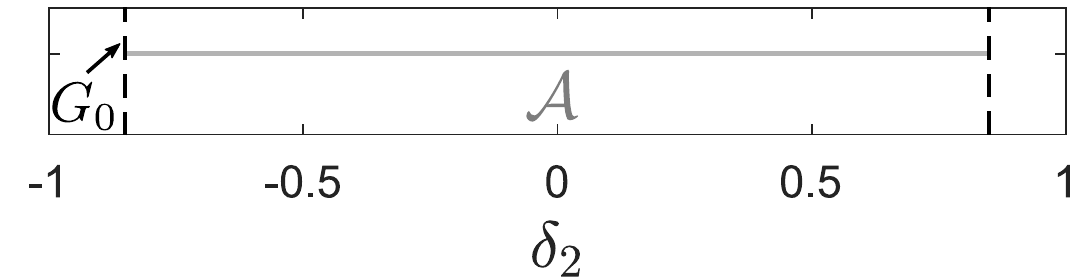} 
		\vspace{-1mm}
		\subcaption{$ |\delta_2| \leq \frac{\pi}{3.7}$}
		\vspace{1mm}
	\end{subfigure}
	\begin{subfigure}[b]{1\linewidth}
		\centering
		\includegraphics[width=7cm,keepaspectratio]{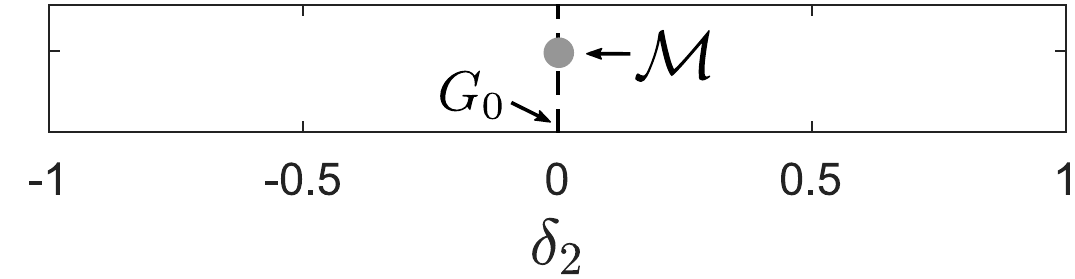} 
		\vspace{-1mm}
		\subcaption{$\delta_2=0$}
	\end{subfigure}
	\caption{The sets for the two-bus system load node with a fixed disturbance bound $\overline{\delta}_1= \frac{\pi}{2}$ and different constraints for ${\delta}_2$.}
	\label{fig:twobus_A&M_L}   
	\vspace{-5mm}
\end{figure}
\subsection{Multi-machine system}
This example shows the set-based analysis for the power networks components with several neighbouring nodes influencing each other.
We consider a six bus post-fault system with a complete graph structure, consisting of four generator nodes, one load node and one reference node.
Therefore each considered component has 5 neighbouring nodes. 
A sketch of the considered power network is shown in Figure~\ref{fig:sixbus_grid}
\begin{figure}[h]
	\begin{center}\
		\includegraphics[width=0.95\linewidth,keepaspectratio]{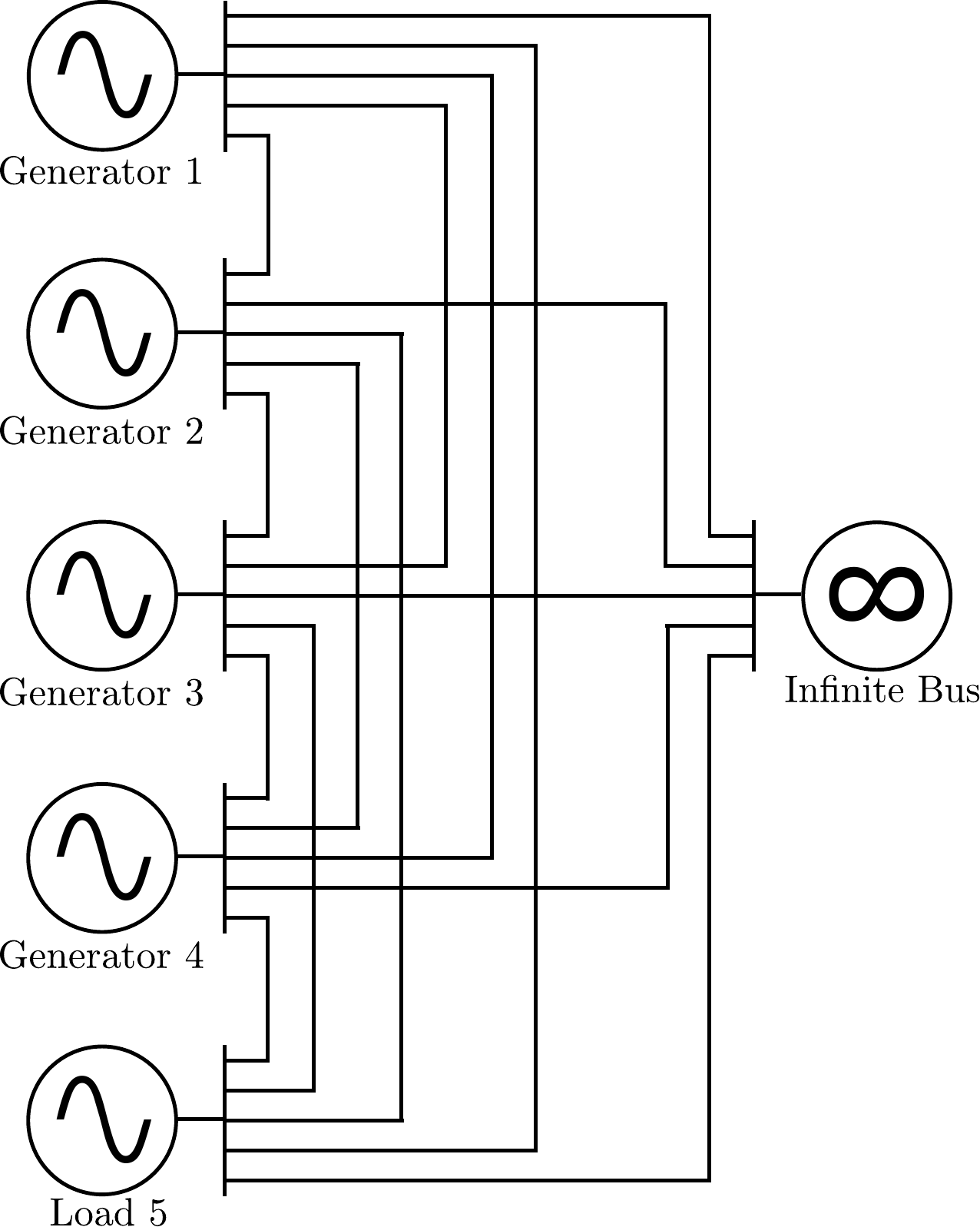} 
		\caption{Sketch of power network with four generator nodes and load bus and a reference node.} 
		\label{fig:sixbus_grid}                             
	\end{center}        
\end{figure}

We use \eqref{eq_generator} for all $i\in\mathcal{G}$ ($i=1,...,4$) and \eqref{eq_load} for $i\in\mathcal{L}$ ($i=5$).
The grid components parameters are as follow $m_i = 1$ and $P_{m_i} = 0.1$, $\forall i \in \mathcal{G}$;
$P_{d_i} = 0.4$, $\forall i \in \mathcal{L}$;
$k_1=0.1$, $k_2=1$, $k_3=2$, $k_4=3$, $k_5=4$; 
$a_{ij}=0.2$, $\forall i \in {1,...,5}$, $\forall j \in {1,...,5}$, $j \not= i$, $a_{i6}=2$, $\forall i \in {1,...,5}$.
We define the constraints for all grid components $|\delta_i|\leq \frac{\pi}{2}$, $\forall i \in {1,...,5}$, from where we deduce the disturbance bounds $D_i^j\triangleq[-\frac{\pi}{2},\frac{\pi}{2}]$, $\forall i \in {1,...,5}$, $\forall j \in {1,...,5}$, $j \not= i$. 
We know from Section \ref{subsec_G_dbarUTC} that $(\pm\frac{\pi}{2},0)$ are the points of ultimate tangentiality for all generator nodes.
We obtain the sets for generator nodes via backwards integration of \eqref{eq_generator} and the corresponding adjoint system \eqref{eq_adj} from these points, using \eqref{dist_cases_invbarrier} and \eqref{dist_cases_barrier} as the disturbance realization.
The results are shown in Figure \ref{fig:6Bus_Generators_all}.
Clearly we can see the influence of different damping parameters $k_i$ for each machines sets, which can lead to an empty MRPI, as can be seen for generator 1.
From our set analysis we clearly identify generator 1 as a critical node, since there is no initial condition from which it is guaranteed that the constraints will be satisfied for all neighbouring node influences.
Therefore the grid operator should focus on protective measures for this node particularly.
\begin{figure}[h]
	\begin{center}\
		\includegraphics[width=0.95\linewidth,height=15cm,keepaspectratio]{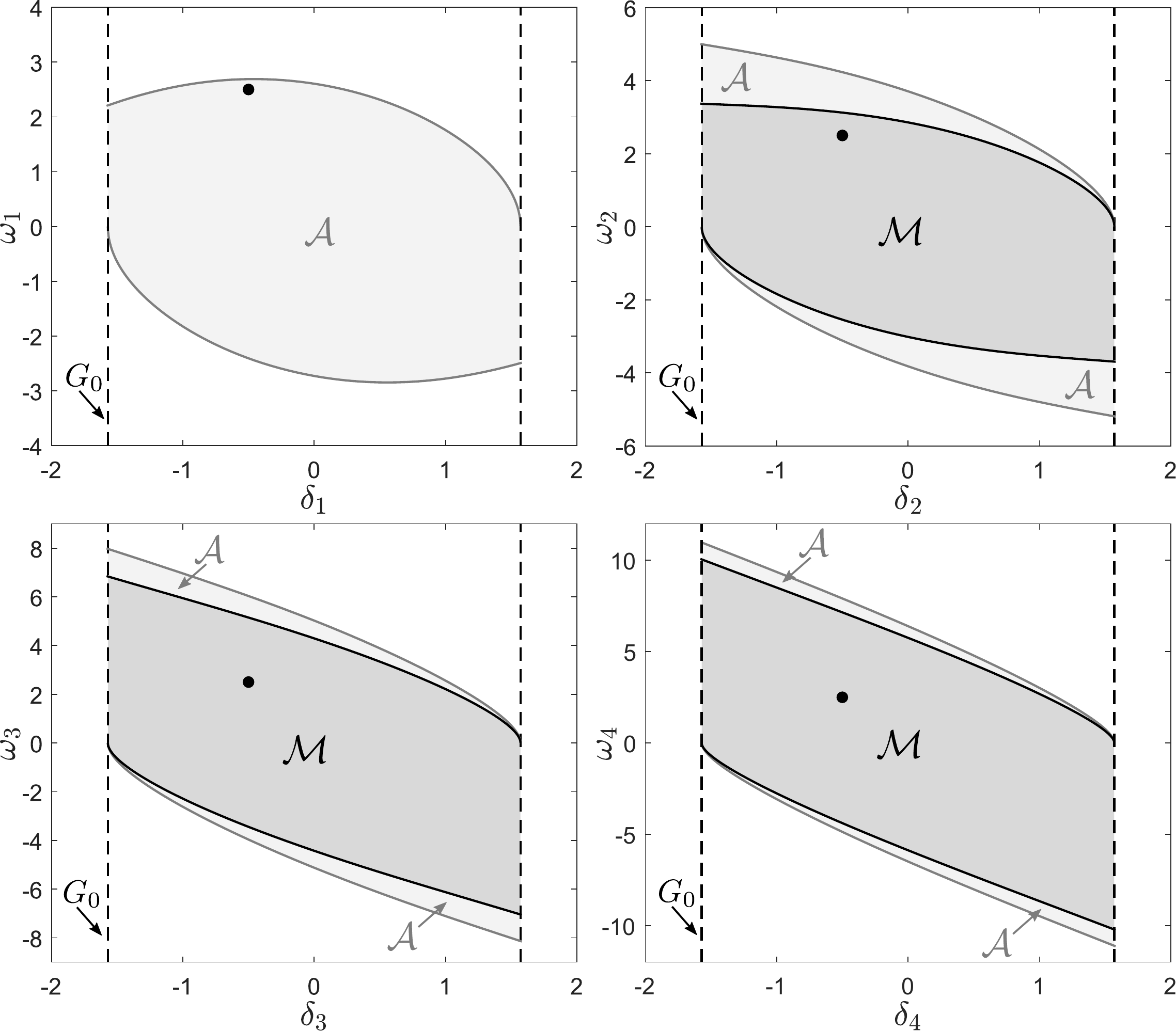} 
		\caption{Sets of all generator nodes in the six bus system with a black dot as initial state of the post fault system. Each considered generator node has a different damping parameter.}
		\label{fig:6Bus_Generators_all}                             
	\end{center}        
\end{figure}
The load bounds are computed via \eqref{eq_load_condlbM} - \eqref{eq_load_condubA} with the result of a robust positively invariant interval in $G$, as shown in Figure~\ref{fig:6Bus_load}.
\begin{figure}[h]
	\begin{center}\
		\includegraphics[width=7cm,keepaspectratio]{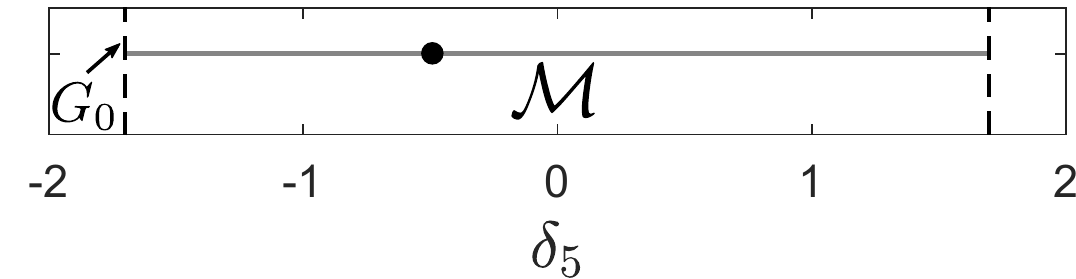} 
		\caption{Set of the load node in the six bus system with a black dot as initial state of the post fault system.} 
		\label{fig:6Bus_load}                             
	\end{center}        
\end{figure}

We simulated the post-fault system with arbitrarily chosen initial conditions marked with a black dot in Figure~\ref{fig:6Bus_Generators_all} and Figure~\ref{fig:6Bus_load}.
The evolving trajectories of the grid components are shown in Figure \ref{fig:6Bus_Trajectories_time}.
Here we can see that our critical identified node (generator 1) will violate the angle constraints without any additional protective measures.
The incorporation of excitation systems is a possible measure to influence the systems transient behavior \cite{AndersonFouad2003_PowerSysControlStab}. 
One of the main factors that affect the transient performance of a generator are the parameters.
The generator excitation control effect can be reduced into an augmented damping parameter $k_i$ \cite{dorfler2012synchronization}.
Since, we used different numerical values for the damping parameter of each component in this example, the effect of a parameter change related to the damping is shown in Figure \ref{fig:6Bus_Generators_all}.
We also see that all trajectories initiating in the safe set $\M$ will not violate the constraints, which is necessary to assess safe grid operations in transient stability.
\begin{figure}[h]
	\begin{center}\
		\includegraphics[width=6.5\linewidth,height=6.5cm,keepaspectratio]{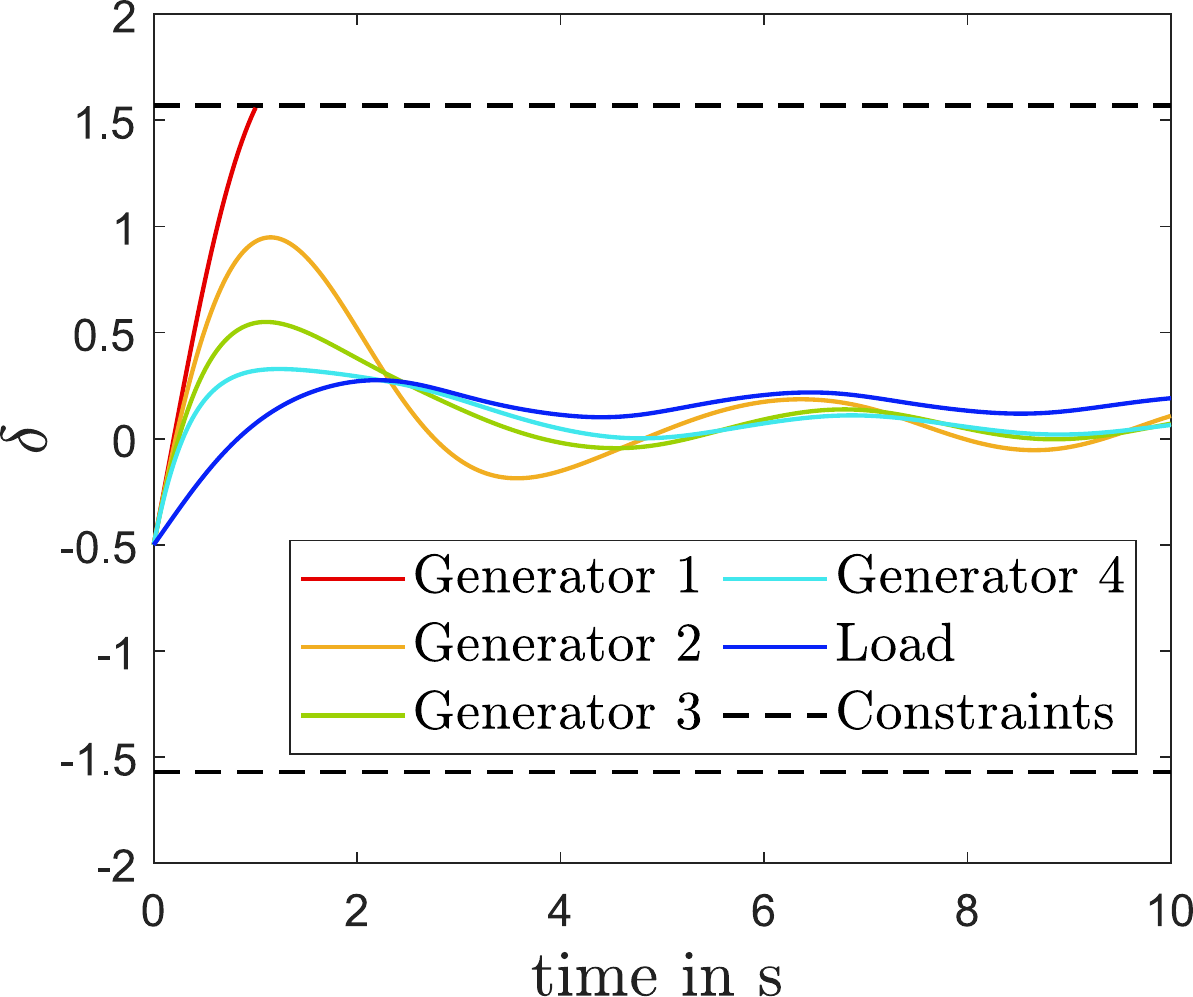} 
		\caption{Time behavior of the generator and load nodes angle in the six bus system.} 
		\label{fig:6Bus_Trajectories_time}                             
	\end{center}        
\end{figure}

\section{Conclusion}\label{section:conclusion}

The main goal of this paper was to assess transient stability of power systems through a new set-based approach involving the admissible set and the MRPI.
We decomposed the overall system, considered each generator and load bus separately and treated the states of neighbouring nodes as disturbances.
We used the theory of \emph{barriers} and \emph{invariance barriers} to construct the admissible set and the MRPI for the system component dynamics.
The results of this paper have been illustrated by a number of examples with varying disturbance bounds and numerical changes of a model parameter showing the different effects on the sets.
More work will need to be done to determine quantitative and qualitative changes of the presented sets with changing constraints and model parameters.

\ifCLASSOPTIONcaptionsoff
  \newpage
\fi



\bibliographystyle{IEEEtran}
\bibliography{Literature_TransStab_AT}
\end{document}